\documentclass[11pt,reqno]{article}

\usepackage[english]{babel}

\usepackage{sectsty}
\sectionfont{\large\bfseries}
\subsectionfont{\normalsize\bfseries}
\subsubsectionfont{\normalsize\mdseries\itshape}

\usepackage{amssymb}
\usepackage{amsmath}
\usepackage{amsthm}

\usepackage{txfonts}
\usepackage{eucal}

\usepackage{enumerate}
\usepackage[longnamesfirst,numbers]{natbib}
\usepackage[bookmarks,colorlinks=true,citecolor=blue]{hyperref}

\theoremstyle{plain}
\newtheorem{thm}{Theorem}
\newtheorem{lem}{Lemma}

\newtheorem*{cor*}{Corollary}
\newtheorem{prop}{Proposition}
\newtheorem*{prop*}{Proposition}

\theoremstyle{definition}
\newtheorem{defn}{Definition}
\newtheorem*{defn*}{Definition}

\theoremstyle{remark}

\numberwithin{equation}{section}

\newcommand{\Bit}{\ensuremath{\{0,1\}}}
\newcommand{\Nat}{\ensuremath{\mathbb{N}}}

\newcommand{\Rat}{\ensuremath{\mathbb{Q}}}
\newcommand{\Real}{\ensuremath{\mathbb{R}}}
\newcommand{\Cant}{\ensuremath{2^{\omega}}}
\newcommand{\Baire}{\ensuremath{\omega^{\omega}}}
\newcommand{\Str}[1][<\omega]{\ensuremath{2^{#1}}}
\newcommand{\Nstr}[1][<\omega]{\ensuremath{\omega^{#1}}}
\newcommand{\Sle}{\ensuremath{\subset}}
\newcommand{\Sleq}{\ensuremath{\subseteq}}

\newcommand{\Sgeq}{\ensuremath{\supseteq}}

\newcommand{\Cl}[1]{\ensuremath{#1}}
\newcommand{\Subcl}[1]{\ensuremath{\Cl{#1} \subseteq \Cant}}
\newcommand{\Cyl}[1]{\ensuremath{N_{#1}}}
\newcommand{\Acyl}[1]{\ensuremath{N(#1)}}

\newcommand{\Rest}[1]{\ensuremath{\!\restriction_{#1}}}

\newcommand{\Conc}{\ensuremath{\mbox{}^\frown}}
\newcommand{\Estr}{\ensuremath{\epsilon}}
\newcommand{\Tup}[1]{\ensuremath{\langle #1 \rangle}}

\newcommand{\Hmeas}{\ensuremath{\mathcal{H}}}
\newcommand{\Hm}[1]{\ensuremath{\Hmeas^{#1}}}

\newcommand{\Hdim}[1][\mbox{}]{\ensuremath{\dim^{#1}_{\operatorname{H}}}}

\newcommand{\Pmeas}{\ensuremath{\CMcal{P}}}
\newcommand{\Leb}{\ensuremath{\lambda}}

\DeclareMathOperator{\T}{T}

\DeclareMathOperator{\Pre}{Pre}
\DeclareMathOperator{\K}{K}
\DeclareMathOperator{\meas}{meas}
\DeclareMathOperator{\supp}{supp}

	\title{Effectively Closed Sets of Measures and Randomness}

	\author{Jan Reimann\thanks{\texttt{reimann@math.berkeley.edu}} \\[2ex] {\small Department of Mathematics}\\ {\small University of California at Berkeley}}

	\date{}

\begin{document}

\maketitle

	\begin{abstract}	
	  We show that if a real $x \in \Cant$ is strongly Hausdorff $\Hm{h}$-random, where $h$ is a dimension function corresponding to a convex order, then it is also random for a continuous probability measure $\mu$ such that the $\mu$-measure of the basic open cylinders shrinks according to $h$. The proof uses a new method to construct measures, based on effective (partial) continuous transformations and a basis theorem for $\Pi^0_1$-classes applied to closed sets of probability measures. We use the main result to give a new proof of Frostman's Lemma, to derive a collapse of randomness notions for Hausdorff measures, and to provide a characterization of effective Hausdorff dimension similar to Frostman's Theorem.
	\end{abstract}

%
%
\section{Introduction} \label{sec-intro}
The duality between measures and the sets they ``charge'' is a central theme in modern analysis. A particular good example of this duality is \emph{fractal geometry}. While it initially studied fractal properties of \emph{sets} in Euclidean or other metric spaces, the geometric analysis is now widely applied to \emph{measures}, too. The books by \citet{mattila:1995} and \citet{edgar:1998} reflect this quite well.

A cornerstone of this development was the work by \citet{frostman:1935}. He realized that there is a close connection between the Hausdorff dimension of a set and the energies of measures residing on it. If a (Borel) set $A$ is large in the sense that its Hausdorff dimension exceeds $s$, then there is a probability measure that resides on $A$ such that its $s$-energy is finite.

\medskip
This article develops an effective analogue to Frostman's work. The main theorem shows that if a real $x$ is strongly random for a Hausdorff measure $\Hm{h}$, there exists a probability measure $\mu$ such that the measure of the basic open cylinders shrinks according to $h$ and such that $x$ is $\mu$-random.

The paper should be seen as an instance or further starting point of a more general endeavor, namely the investigation of how the complexity of a real (logical or randomness theoretic) relates to the complexity of the measures for which the real is random. The idea is that the study of duality between the complexity of sets and measures leads to interesting insights and questions when transferred to an effective setting. Recent work by Theodore Slaman and the author \citep{reimann-slaman:tbs,reimann-slaman:ip2} gives evidence for this.

While a real is non-trivially random (i.e.\ there exists a probability measure such that $x$ is $\mu$-random and $\mu(\{x\}) = 0$, see \citep{reimann-slaman:tbs}) if and only if it is not recursive, the question which reals are random for continuous probability measures found quite unexpected answers. In particular, every real that is not random for such a measure with respect to $\Sigma^0_1$-tests is $\Delta^1_1$. However, a complete characterization of such reals is yet unknown. 
This paper can also be seen as a contribution to this question, by excluding uncountable $\Pi^0_1$ classes of nontrivial Hausdorff measure and giving a sufficient criterion for randomness for a continuous measure.

\medskip
The techniques used to prove the main result are quite different from the classical setting. In the latter, tools like the Hahn-Banach theorem and the properties of weak convergence play a major role. In the effective setting, these are replaced rather by recursion theoretic techniques such as basis theorems for $\Pi^0_1$ classes and the Gacs-Kucera reducibility result.
The title of this article refers to applying these techniques to sets of (representations of) measures instead of sets of reals. As an easy corollary, we obtain a new proof of the classical result based on the effective techniques.

This raises the hope that the effective theory might in turn contribute to the classical setting, via relativization, as it did in the case of effective descriptive set theory. Recently, \citet{kjoshanssen:ip, kjoshanssen:ip2} has obtained results in this direction.

The outline of the paper is as follows. In Section \ref{sec-meas-top-cant}, we first present a brief account of measure theory on Cantor space. The basic notion is that of an outer measure derived from a premeasure. We will then put particular emphasis on defining a representation of the space of probability measures that is a $\Pi^0_1$ subset of $\Cant$. Finally, we introduce randomness for arbitrary outer measures.

Section \ref{sec-rand} presents the main results. We first deal with basis theorems for $\Pi^0_1$ sets (of measures). After giving some background from the classical theory, we prove that every real that is strongly Hausdorff random is also random with respect to a certain continuous probability measure. 

Section \ref{sec-consequ-cap} presents a number of applications of the main result. First, we give a new proof of Frostman's Lemma. Then we investigate randomness notions for Hausdorff measures. Several test notions based on uniformly r.e.\ tests have been proposed, most of which are equivalent on probability measures but differ on Hausdorff measures. We use the relation between randomness for Hausdorff measures and continuous probability measures to show that in a certain sense strong randomness is indeed the strongest possible randomness notion.  
As an easy corollary, we obtain that the dimension notions induced by such randomness notions coincide.
Finally, we use the main theorem to derive an effective version of Frostman's result, showing that Hausdorff and capacitary dimension coincide. 

We conclude with a few remarks on analogous results for packing dimension, and discuss an open question as well as directions for further research.

We assume familiarity with the basic notions of measure theory and descriptive set theory, as treated in \citep{kechris:1995}. Furthermore, we presuppose some knowledge in algorithmic randomness and computability theory, as can be found in \citep{li-vitanyi:1997} or \citep{downey-hirschfeldt:ip}.

%
%
\section{Measures and Randomness on Cantor Space} \label{sec-meas-top-cant}

In this section we briefly review the basic notions of measure on the Cantor space $\Cant$. We make use of the special topological structure of $\Cant$ to give a unified treatment of a large class of measures, not necessarily $\sigma$-finite. 
This way, we do not have to distinguish between probability measures and Hausdorff measures, for instance.

%
%
\subsection{Outer measures on Cantor space}

We work in $\Cant$ as a compact Polish space. A countable basis of the topology is given by the \emph{cylinder sets}
\[
	N_\sigma = \{ x : \: x\Rest{n} = \sigma\},
\]
where $\sigma$ is a finite binary sequence. We will occasionally use the notation $\Acyl{\sigma}$ in place of $\Cyl{\sigma}$ to avoid multiple subscripts. $\Str$ denotes the set of all finite binary sequences. If $\sigma, \tau \in \Str$, we use $\Sleq$ to denote the usual prefix partial ordering. This extends in a natural way to $\Str \cup \Cant$. Thus, $x \in \Cyl{\sigma}$ if and only if $\sigma \Sle x$. Finally, given $U \subseteq \Str$, we write $\Cyl{U}$ to denote the open set induced by $U$, i.e. $\Cyl{U} = \bigcup_{\sigma \in U} \Cyl{\sigma}$.

The following method to construct outer measures has been referred to as \emph{Method I} \citep{munroe:1953, rogers:1970}.

\begin{defn}
Let $\Str$ be the set of all finite binary sequences. A \emph{premeasure} is a mapping $\rho: \Str \to \mathbb{R}^{\geq 0}$.
\end{defn}
	
If $\rho$ is a premeasure, define the set function $\mu_\rho: \mathcal{P}(\Cant) \to \mathbb{R}^{\geq 0}$ by letting
	\begin{equation} \label{equ-outer-measure}
		\mu_\rho(A) = \inf \left\{ \sum_{\sigma \in U} \rho(\sigma) : \: A \subseteq  N_U \right\},
	\end{equation}
where we set $\mu_\rho(\emptyset) = 0$.
It can be shown that $\mu_\rho$ is an outer measure, i.e.\ a real-valued, non-negative, monotone, subadditive set function. A \emph{measure} is given by restricting an outer measure to the $\sigma$-algebra of \emph{measurable sets}. Since we are mostly interested in nullsets, which are always measurable, we do not make the distinction between measures and outer measures here, and will subsequently refer to an outer measure simply as \emph{a measure}.
Furthermore, we will always assume that an outer measure $\nu$ is derived from a premeasure as in \eqref{equ-outer-measure}. (\citet{rogers:1970} studies in great detail the relations between measures, outer measures, and premeasures.)

Of course, the nature of the outer measure $\mu_\rho$ obtained via \eqref{equ-outer-measure} will depend on the nature of the premeasure $\rho$. We will study two important classes of outer measures: probability measures and Hausdorff measures.

%
%
\subsection{Probability measures}

A \emph{probability measure} $\nu$ is any measure that is based on a premeasure $\rho$ which satisfies $\rho(\emptyset) = 1$ and
\begin{equation} \label{equ-probability measure}
	\rho(\sigma) = \rho(\sigma \Conc 0) + \rho(\sigma \Conc 1)
\end{equation}
for all finite sequences $\sigma$. The resulting measure $\mu_\rho$ preserves $\rho$ in the sense that $\mu_\rho(N_\sigma) = \rho(\sigma)$ for all $\sigma$. This follows from the Caratheodory extension theorem. We denote by $\Pmeas$ the set of all probability measures on $\Cant$. In the following, we will often identify probability measures with their underlying premeasure, i.e.\ we will write $\mu(\sigma)$ instead of $\mu(\Cyl{\sigma})$.

If $\rho$ is a probability premeasure, then $\mu_\rho$ is a \emph{Borel measure}, i.e.\ all Borel sets are measurable, and their measure is a finite real number in $[0,1]$. 

It will later be important to identify the subset of $\Cant$ on which a probability measure `resides'. The \emph{support} $\operatorname{supp}(\mu)$ of a probability measure $\mu$ is the smallest closed set $F \subseteq \Cant$ such that $\mu(\Cant\setminus F) = 0$.

\medskip
For $\rho(\sigma) = d(\Cyl{\sigma}) = 2^{-|\sigma|}$ we obtain the \emph{Lebesgue measure} $\Leb$ on $\Cant$, which is the unique translation invariant measure $\mu$ on $\Cant$ for which $\mu(\Cyl{\sigma}) = d(\Cyl{\sigma})$. Here $d$ denotes the diameter function derived from the standard metric on Cantor space,
\[
	d(x,y) = \begin{cases}
		2^{-N} & \text{if $x \neq y$ and $N$ is minimal such that $x(N) \neq y(N)$}, \\
		0 & \text{if $x=y$}.
	\end{cases}
\]

\medskip
\emph{Dirac measures} are probability measures concentrated on a single point. For $x \in \Cant$, we define
\[
	\rho(\sigma) = \begin{cases}
	   	1 & \text{if } \sigma \subset x, \\
		0 & \text{otherwise.}
	\end{cases}
\]
For the induced outer measure we obviously have $\mu_\rho(A) = 1$ if and only if $x \in A$, and $\mu_\rho(A) = 0$ if and only if $x \not\in A$. The corresponding measure is denoted by $\delta_x$.

\medskip
If, for a measure $\mu$ and $x \in \Cant$, $\mu(\{x\}) > 0$, then $x$
is called an \emph{atom} of $\mu$. Obviously, $x$ is the unique atom of
$\delta_x$. A measure that does not have any atoms is called
\emph{continuous}.

%
%
\subsection{Hausdorff measures}

Hausdorff measures are of fundamental importance in geometric measure theory. They share the common feature that the premeasures they stem from only depend on the diameter of an open set. Therefore, the resulting measure will be translation invariant. 

A \emph{dimension function} $h$ is a nonnegative, nondecreasing, continuous on the right function defined on all nonnegative reals. Assume, furthermore, that $h(t) > 0$ if and only if $t>0$. Define the premeasure $\rho_h$ as
\[
	\rho_h(\Cyl{\sigma}) = h(d(\Cyl{\sigma})) = h(2^{-|\sigma|}).
\]

The resulting measure $\mu_{\rho_h}$ will in general not be a Borel measure. Therefore, one refines the transition from a premeasure to an outer measure, also known as \emph{Method II} \citep{munroe:1953, rogers:1970}. The resulting outer measure is denoted by $\Hm{h}$. Again, we will mostly be concerned with $\mathcal{H}^h$-nullsets. It is not hard to see that for any set $A$, $\mathcal{H}^h(A)= 0$ if and only if $\mu_{\rho_h}(A)=0$, that is, the nullsets obtained from a premeasure via Method I and Method II coincide. Hence in the case of nullsets we can work with the less involved definition via Method I and therefore refer the reader to the above references for details on Method II.

Due to the special nature of the standard metric $d$ on $\Cant$, only diameters of the form $2^{-n}$, $n \in \Nat$, appear. So we can take any nondecreasing, unbounded function $h: \Nat \to \Real^{\geq 0}$ (in fact, $h: \Nat \to \Nat$ suffices), and set $\rho_h(\Cyl{\sigma}) = 2^{-h(|\sigma|)}$. Such functions $h$ are called \emph{orders}. The resulting Hausdorff measure will, in slight abuse of notation, also be denoted by $\mathcal{H}^{h}$. Finally an order is called \emph{convex}, if for all $n$, $h(n+1) \leq h(n) + 1$. Let $\Hmeas$ be the set of premeasures corresponding to convex order functions.

\citet{reimann-stephan:2005} studied the class of \emph{geometrical premeasures}. These satisfy the following condition: There 
are real numbers $p,q$ such that
\begin{enumerate}[(G1)]
\item $1/2 \leq p < 1$ and $1 \leq q < 2$;
\item $(\forall \sigma \in \Str) \, (\forall i \in \Bit) \: [\rho(\sigma \Conc i) \leq
p\rho(\sigma)]$;
\item $(\forall \sigma \in \Str) \: [q \rho(\sigma) \leq \rho(\sigma \Conc 0)+\rho(\sigma \Conc 1)]$.
\end{enumerate}

The class $\mathcal{G} \subseteq  \Pmeas \, \cup\, \Hmeas$ of all geometrical premeasures comprises all premeasures based on orders of the form $h(n) = \alpha n$, $0 < \alpha \leq 1$, as well as many probability measures such as all non-degenerate Bernoulli measures and measures satisfying \eqref{equ-frostman-prop} for a geometrically increasing order.

\medskip
Among the numerous Hausdorff measures, the family of $t$-dimensional Hausdorff measures $\Hm{t}$ given by $h(n) = tn$, where $0\leq t \leq 1 $ is arguably most prominent. It is not hard to see that for any set $A$, $\Hm{s}(A) < \infty$ implies $\Hm{t}(A) = 0$ for all $t > s$. Likewise, $\Hm{r}(A) = \infty$ for all $r < s$. Thus there is a critical value where $\Hm{s}$ `jumps' from $\infty$ to $0$. This value is called the \emph{Hausdorff dimension} of $A$, written $\Hdim A$. Formally,
\[
	\Hdim A = \inf \{ s : \: \Hm{s}(A) = 0 \}.
\]
Hausdorff dimension is a central notion in fractal geometry and has recently received a lot of attention in the effective setting, too. We will not dwell further on this here but refer to  \citep{falconer:1990, mattila:1995, lutz:2003, reimann:2004}.

%
%
\subsection{Representations of premeasures} \label{subsec-repres-premeas}

To define randomness, we want to incorporate measures into the effective aspects of a randomness test. For this purpose, we have to represent it in a form that makes it accessible for recursion theoretic methods. Essentially, this means to code a measure via an infinite binary sequence or a function $f:\Nat \to \Nat$.

\medskip
The way we introduced it, an outer measure on $\Cant$ is completely determined by its underlying premeasure defined on the cylinder sets. It seems reasonable to represent these values via approximation by rational intervals.

\begin{defn}\label{def-rational-repr}
	Given a premeasure $\rho$, define its \emph{rational representation} $r_\rho$ by letting, for all $\sigma \in \Str$, $q_1, q_2 \in \Rat$,
	\begin{equation}
		\Tup{\sigma, q_1, q_2} \in r_\rho \; \Leftrightarrow \; q_1 < \rho(\sigma) < q_2.
	\end{equation}
\end{defn}
 
The real $r_\rho$ encodes the complete information about the premeasure $\rho$ in the sense that  for each $\sigma$, the value $\rho(\sigma)$ is uniformly recursive in $r_\rho$. 

\medskip
While the rational representation is defined for every premeasure, it does not reflect well the richer structure of certain families of measures, such as the set of probability measures. In particular, the set of reals $z$ such that $z$ is the rational representation of a probability measure is not $\Pi^0_1$. Since we will need to exploit this structure later, in the following we introduce an alternative representation for probability measures.

An effective representation of probability measures has been developed elsewhere (for example in \citep{gacs:2005}), but we found none of the accounts completely adequate for our purposes. Therefore, in the following we give a succinct description of how to devise a $\Pi^0_1$-class in $\Cant$ of representations of probability measures. In general, our approach follows the framework of \citet{moschovakis:1980}, with a few adaptations regarding compactness. Each step can be justified by resorting to the accordant results in \citep{moschovakis:1980}.

%
%
\subsection{The space of probability measures on Cantor space}

Recall that $\Pmeas$ denotes the set of all probability measures on $\Cant$. $\Pmeas$ can be given a topology (the so-called \emph{weak topology}) by defining $\mu_n \to \mu$ if $\mu_n(\Cl{B}) \to\mu(\Cl{B})$ for all Borel sets $\Cl{B}$ whose boundary has $\mu$-measure $0$. This is equivalent to requiring that $\int f d\mu_n \to \int f d\mu$ for all bounded continuous real-valued functions $f$ on $\Cant$. In Cantor space, the weak topology is also generated by sets of the form
\begin{equation} \label{equ-rat-repres}
	M_{\sigma, p,q} = \{ \mu \in \Pmeas: \: p < \mu(\Cyl{\sigma}) < q \},
\end{equation}
where $\sigma \in \Str$ and $p < q$ are rational numbers in the unit interval. The sets $M_{\sigma, p,q}$ form a \emph{subbasis} of the weak topology.
 
It is known that if $X$ is Polish and compact metrizable, then so is the space of all Borel  probability measures on $X$ (see for instance \citep{kechris:1995}). Therefore, 
$\Pmeas$ is compact metrizable and Polish. A compatible metric is given by
\[
	d_{\meas}(\mu,\nu) = \sum_{n=1}^\infty 2^{-n} d_n(\mu,\nu),
\]
where
\[
	d_n(\mu,\nu) = \frac{1}{2} \sum_{|\sigma| = n} |\mu(\Cyl{\sigma}) - \nu(\Cyl{\sigma})|.
\]

We want to find an \emph{effective presentation} of $\Pmeas$ reflecting its properties. A countable, dense
subset $\mathcal{D} \subseteq \Pmeas$ is given by the set of measures which assume
positive, rational values on a finite number of rationals, i.e. $\mathcal{D}$ is the set
of measures of the form
\begin{equation*}
  \nu_{\Delta, Q} = \sum_{\sigma \in \Delta} Q(\sigma) \delta_{\sigma \Conc 0^\omega},
\end{equation*}
where $\Delta$ is a finite set of finite strings (representing dyadic rational numbers) and $Q: \Delta \to [0,1]\cap \Rat$ such that $\sum_{\sigma \in \Delta} Q(\sigma) = 1$.  

A straightforward calculation shows that for $\mu, \nu \in \mathcal{D}$, the following two relations are recursive:
\[
	d_{\meas}(\mu,\nu) < q \quad \text{ and } \quad d_{\meas}(\mu,\nu) \leq q, 
\]
for rational $q \geq 0$.

We fix an effective enumeration of $\mathcal{D} = \{\nu_0, \nu_1, \dots\}$. $\mathcal{D}$ is called a \emph{recursive presentation}. We can invoke the basic machinery of \emph{(effective) descriptive set theory} to obtain a recursive surjection
\[
	\pi: \Baire \to \Pmeas 
\]
and a $\Pi^0_1$ set $P \subseteq \Baire$ such that $\pi|_P$ is one-to-one and $\pi(P) = \Pmeas$. 
(See \citet{moschovakis:1980}, 3E.6.) This is achieved via an \emph{effective Lusin scheme}, a family of sets mirroring the tree structure of $\Baire$. 

However, since $\Pmeas$ is compact, we would like to have a representation that is not only closed but \emph{effectively compact}, that is, a $\Pi^0_1$ subset of $\Cant$. We can obtain such a representation, but we have to give up injectivity for this.
We can use the compactness of $\Pmeas$ to devise a finitely branching \emph{Souslin scheme} for $\Pmeas$. 
More precisely, there exists a finite branching $T \subset \Nstr$ and a family $(U_\sigma)_{\sigma \in T}$ of non-empty open sets in $\Pmeas$ such that 
\begin{quote}
	\begin{enumerate}[(1)]
	\item $U_\Estr = \Pmeas$,
	\item for all $\sigma \in T$, $U_{\sigma} = \bigcup_{\sigma\Conc i \in T} U_{\sigma\Conc i}$,
	\item for all $\sigma\Conc i \in T$, $\overline{U_{\sigma\Conc i}} \subseteq U_\sigma$,
	\item for $\sigma \in T$, $d_{\meas}(U_\sigma) \leq 2^{-|\sigma|}$.
	\end{enumerate}	
\end{quote}
We can even assume that $T$ is \emph{uniformly branching}, i.e.\ for every level $n$, the number of immediate extensions of a string of length $n$ in T is unique. 
Since $\Pmeas$ is complete, for every $p \in [T]$, 
\[
	\bigcap_n U_{p\Rest{n}} \neq \emptyset.	
\]
In fact, $\bigcap_n U_{p\Rest{n}}$ contains a single measure $\mu_p$. Hence we can call $p$ a representation of $\mu_p$. Note that since the $U_{\sigma \Conc i}$ are not necessarily disjoint, a measure can have multiple (in fact, continuum many) representations this way.

The $U_\sigma$ can be chosen as open balls with respect to $d_{\meas}$ centered on measures from $\mathcal{D}$. More precisely, we can use the effectiveness of the metric $d_{\meas}$ to find a recursive, increasing sequence $(l_n)$ of natural numbers such that 
\begin{enumerate}[(i)]
	\item $U_\sigma$ is of the form $B_{2^{-n}}(\mu)$, where $\mu\in \mathcal{D}$ is concentrated on strings of length $l_n$ (or, more formally, on reals of the form $\sigma\Conc 0^\omega$, where $|\sigma| = l_n$),
	\item all sets of measures of this form appear as a $U_\sigma$.
	\item the mapping $\sigma \mapsto (\Delta,Q)$ such that $U_\sigma = B_{2^{-n}}(\nu_{\Delta, Q})$ is computable.
\end{enumerate}
Summing up, we obtain the following representation.

\begin{prop}
	There exists a recursive sequence $(r_n)$ and a continuous surjection
	\[
		\pi: [T] \to \Pmeas,
	\]
	where $T \subset \Nstr$ is the full $(r_n)$-branching tree, i.e.\ every node in $T$ of length $n$ has exactly $r_n$ immediate successors.
\end{prop}

Now $P = [T]$ is a compact subset of $\Baire$. By a standard embedding, this corresponds to a closed subset of $\Cant$. We can therefore assume that $P$ is a $\Pi^0_1$ subset of $\Cant$, represented by a recursive, pruned tree $T_P \subseteq \Str$ such that $[T_P] = P$.

Every element in $P$ is an intersection of the nested $U_\sigma$, corresponding to a path through $T$. This path in turn represents a \emph{Cauchy sequence} of measures in $\mathcal{D}$.
We will therefore call the representation given by $P$ the \emph{Cauchy representation} of $\Pmeas$.

%
%
\subsection{Randomness} \label{ssec-randomness}

We briefly review the definition of randomness in the sense of Martin-L\"of for arbitrary outer measures. We refer to \citep{reimann:ta,reimann-slaman:tbs} for more details on this approach to randomness for arbitrary measures.

Martin-L\"of's concept of randomness is based on the fact that every nullset for a measure defined via Method I (and Method II, as is easily seen) is contained in a $G_\delta$-nullset. Essentially, a Martin-L\"of test is an effectively presented $G_\delta$ nullset (relative to some parameter $z$). 

\begin{defn} \label{def-test}
	Suppose $z \in \Cant$ is a real. A \emph{test relative to $z$}, or simply a \emph{$z$-test}, is a set $W \subseteq \Nat \times \Str$ which is recursively enumerable in $z$. A real $x$ \emph{passes} a test $W$ if $x \not\in \bigcap_n \Acyl{W_n}$, where $W_n = \{ \sigma : \: (n, \sigma) \in W \}$.
\end{defn}

Passing a test $W$ means not being contained in the $G_\delta$ set given by $W$. The condition `\emph{r.e.\ in $z$}' implies that the open sets given by the sets $W_n$ form a uniform sequence of $\Sigma^0_1(z)$ sets, and the set $\bigcap_n \Acyl{W_n}$ is a $\Pi^0_2(z)$ subset of $\Cant$. To test for randomness, we want to ensure that $W$ actually describes a nullset.

\begin{defn} 
Suppose $\rho$ is a premeasure on $\Cant$. A $z$-test $W$ is \emph{correct for $\mu_\rho$} if 
		\begin{equation} \label{equ-correct-test}
			\sum_{\sigma \in W_n} \rho(\Cyl{\sigma}) \leq 2^{-n}.
		\end{equation}
	Any test which is correct for $\mu_\rho$ will be called a \emph{$z$-test for $\mu_\rho$}, or $\mu_\rho$-$z$-test. (As always, we will drop the subscript $\rho$ if the premeasure is clear from the context.) 
\end{defn}

Finally, we incorporate the information given by the (representation of the) premeasure into the test notion to define randomness with respect to arbitrary outer measures.

\begin{defn}\label{def-test-relative-to-measure}
	Suppose $\rho$ is a premeasure on $\Cant$ and $z \in \Cant$ is a real. Let $p_\rho$ be a representation of $\rho$, i.e.\ either the rational representation $r_\rho$, or in case of a probability measure a Cauchy representation $p \in P$ such that $\pi(p) = \mu_\rho$.
A real is \emph{$\mu_\rho$-random relative to $z$ and representation $p_\rho$}, if it passes all $p_\rho \oplus z$-tests which are correct for $\mu_\rho$.
\end{defn}

An obvious objection to this definition of randomness is that it is \emph{representation dependent}. In fact, \citet{levin:1976,levin:1984} and recently \citet{gacs:2005} have given definitions of randomness independent of the underlying measure. There are arguments in favor of and against both approaches (see also \citep{reimann:ta}). In the context of this paper, the major problem with the representation independent approach is that it is quite difficult to make the measures subject to a classification in terms of logical complexity, such as \emph{Turing degrees}, see \citep{miller:2004}.  

\medskip
In the following, when we say that a real is $\mu$-random, we mean that there exists a representation of $\mu$ such that the real is $\mu$-random relative to that representation.

%
%
\section{Effectively closed sets of measures, randomness, and capacities} \label{sec-rand}

%
%
\subsection{Effectively closed sets and randomness} \label{ssec-closed-rand}

\citet{levin:1973} was the first to use $\Pi^0_1$ classes of measures in algorithmic randomness.
He observed that, given a test $W$, the set of probability measures $\mu$ such that $W$ is correct for $\mu$ is $\Pi^0_1$. Levin was interested in devising \emph{uniform tests for randomness}, and he proved the following related result.
 
\begin{thm}[\citet{levin:1973}]
		Given an effectively closed set $S$ of probability measures, there exists a test $U$ such that for any $x$ that passes $U$ there exists a measure $\mu \in S$ such that $x$ is $\mu$-random.
\end{thm}

A result in a similar spirit has recently been shown independently by \citet{downey-hirschfeldt-miller-nies:2005} and \citet{reimann-slaman:tbs}.

%
%
\begin{thm}\label{thm-rand-conserv-basis}
  Let $z \in \Cant$, and let $T \subseteq \Str$ be an infinite
  tree recursive in $z$. Then, for every real
  $R$ which is $\Leb$-random relative to $z$, there is
  an infinite path $y$ through $T$ such that $R$ is $\Leb$-random relative to $z \oplus y$.
\end{thm}

Theorem \ref{thm-rand-conserv-basis} will be the essential ingredient in constructing measures that make a given real random. The basic paradigm for this is:
\begin{quote}
	\emph{Transform a $\Leb$-random real  and find among the admissible transformed measures one that conserves randomness.}
\end{quote}
More precisely, we want to make a real $x$ random. To do so, we Turing reduce it to a $\Leb$-random real. The Turing reduction induces a partial continuous transformation of $\Cant$, which in turn induces a transformation of Lebesgue measure. If we can show that the set of admissible image measures is $\Pi^0_1$, we can use Theorem \ref{thm-rand-conserv-basis} to find one representation that, when added as a parameter to a $\Leb$-test, preserves randomness. 

Subsection \ref{ssec-cap-eff-closed} will present an elaborate example of this technique.
Before, we will motivate this by presenting some background from geometric measure theory.

%
%
\subsection{Capacities and Hausdorff dimension} \label{ssec-cap-dim-class}

The calculation of Hausdorff dimension is often a very difficult task, in particular, obtaining a lower bound. One of the standard tools is the \emph{mass distribution principle} (see \citep{falconer:1990}). A \emph{mass distribution on $A \subseteq \Cant$} is a probability measure such that $\operatorname{supp}(\mu) \subseteq A$.

%
%
\begin{thm} \label{cant:thm_mdistprinc}
Let $A \subseteq \Cant$, and let $\mu$ be a mass distribution on $A$. Suppose further that for some
$s \geq 0$ there are $c, \delta > 0$ such that 
\begin{equation} \label{cant:equ_mdistprop}
        \mu(B) \leq c d(B)^s
\end{equation}
for all $B \subseteq \Cant$ with $d(B) < \delta$. Then 
\[
        s \leq \Hdim(A).
\]
\end{thm}

Theorem \ref{cant:thm_mdistprinc} can be generalized by classifying
mass distributions on $A$ according to whether they satisfy
\eqref{cant:equ_mdistprop} for some $s$. This approach is closely
related to the notion of \emph{capacity}, which first was studied in the
context of potential theory. 
 
%
%
\begin{defn} \label{cant:def_potenergy}
Let $\mu$ be a mass distribution on $\Cant$, $0 \leq t \leq 1$. The
$t$-\emph{potential at} $x \in \Cant$ \emph{due to} $\mu$ is defined
as 
\begin{equation} \label{cant:equ_poten}
        \phi_t(x) = \int d(x,y)^{-t} d\mu(y).
\end{equation}
The $t$-\emph{energy} of $\mu$ is given by
\begin{equation}\label{cant:equ_energy}
        I_t(\mu) = \int \phi_t(x) d\mu(x) = \iint
        d(x,y)^{-t} d\mu(y) d\mu(x). 
\end{equation}
\end{defn}
Observe that if a mass distribution satisfies for some $c,s \in \Real$
\begin{equation}\label{cant:equ_massdbound}
        \mu(\sigma) \leq c2^{-|\sigma|s} \; \text{ for all $\sigma \in \Str$},
\end{equation}
it follows immediately that $\phi_t(A) \leq \operatorname{const}$ for all $t < s$,
hence $I_t(\mu) < \infty$. On the other hand, if $I_t(\mu) < \infty$,
\eqref{cant:equ_massdbound} holds for a suitable restriction of
$\mu$. 

%
%
\begin{defn}\label{cant:def_capac}
Let $s > 0$. The $s$-\emph{capacity} of a class $\Subcl{A}$ is defined as
\[
        C_s(\Cl{A}) = \sup \left\{ \frac{1}{I_s(\mu)}: \: \mu \text{
            mass distr. on $\Cl{A}$ with $\mu(\Cant) = 1$} \right\}.
\]
\end{defn}

(As potentials and capacities may be infinite, we adopt the convention
that $1/\infty = 0$.) We note from the definition that a set has
positive $s$-capacity if and only if there is a mass distribution
$\mu$ on it such that $I_s(\mu) < \infty$. This suggests the following
definition. 

%
%
\begin{defn}\label{cant:def_capdim}
The \emph{capacitary dimension} of a class $\Subcl{A}$ is
\[
        \dim_c(\Cl{A}) = \sup \{ s: \: C_s(\Cl{A}) > 0 \}.
\]
\end{defn}
With little effort it can be shown that 
\[
        \dim_c(\Cl{A}) = \sup \{ s: \: \exists c\: \exists \mu \text{ mass
          distr. on $\Cl{A}$ with $(\forall \sigma \in \Str)\:[\mu(\sigma) \leq c\,2^{-|\sigma|s}]$ } \}. 
\]
Furthermore, the capacitary dimension of a Borel set is equal to its
Hausdorff dimension. 

%
%
\begin{thm}[\citet{frostman:1935}] \label{cant:thm_cdimeqhdim}
Let $\Cl{A}$ be a Borel set in $\Cant$. Then
\[
        \dim_c(\Cl{A}) = \Hdim(\Cl{A}).
\]
\end{thm}

%
%
\subsection{Capacitability of strongly complex reals} \label{ssec-cap-eff-closed}

We will now show that every strongly ${h}$-complex real is ``effectively capacitable''. Note that $\Hm{h}$-almost every real is strongly $h$-complex. We will see later that strong complexity is actually a necessary assumption to prove effective capacitability, hence the following result completely describes to what extend Frostman's Lemma holds effectively.

First, we introduce strong $h$-complexity. \citet{kjos-hanssen-merkle-stephan:2006} defined a real to be \emph{complex} if there exists a computable order function $h$ such that 
\begin{equation}\label{equ-complex}
	(\forall n)\: [ \K(x\Rest{n}) \geq h(n)],
\end{equation}
where $\K$ denotes prefix-free Kolmogorov complexity.
If $x$ is complex via $h$, then we call $x$ $h$-complex. \citet{reimann:2004} showed that $x$ is $h$-complex if and only if it is $\Hm{h}$-random.

One can introduce variants of complexity for reals by replacing $\K$ in \eqref{equ-complex} by another type of Kolmogorov complexity. A \emph{(continuous) semimeasure} is a function $\eta: \Str \to [0,1]$ such that
\begin{equation} \label{equ-semimeas}
	\forall \sigma \: [ \eta(\sigma) \geq \eta(\sigma \Conc 0) + \eta(\sigma\Conc 1) ].
\end{equation}

Levin \citep{zvonkin-levin:1970} proved the existence of an \emph{optimal enumerable semimeasure} $\overline{M}$. A semimeasure is \emph{enumerable} if the set $\{(\sigma,q) \in \Str\times\Rat\colon q < \eta(\sigma)\}$ is r.e. For any enumerable semimeasure $\eta$ there exists a constant $c$ such that for every $\sigma$,
\[
	\eta(\sigma) \leq c \overline{M}(\sigma).
\]
The \emph{a priori complexity} of a string $\sigma$ is defined as $-\log \overline{M}(\sigma)$.
Given a computable order $h$, we say a real $x\in \Cant$ is \emph{strongly $h$-complex} if  
\begin{equation}\label{equ-str-complex}
	(\forall n)\: [ -\log \overline{M}(x\Rest{n}) \geq h(n)],
\end{equation}
Note that up to a constant, $-\log \overline{M}(\sigma) \leq \K(\sigma)$ for all $\sigma$. Hence \eqref{equ-str-complex} implies \eqref{equ-complex}, which justifies the name \emph{strongly complex}. In particular, every strongly $h$-complex real is $\Hm{h}$-random.

Finally, given an order $h$, we say $x$ is \emph{$h$-capacitable} if there exists a probability measure $\mu$ such that, for some $\gamma$,
\begin{equation} \label{equ-frostman-prop}
		(\forall \sigma) \: [\mu(\sigma) \leq \gamma 2^{-h(|\sigma|)}],
\end{equation}
and such that $x$ is $\mu$-random. In the following, we will call a probability measure $\mu$ satisfying \eqref{equ-frostman-prop} \emph{$h$-bounded}.

\begin{thm}[Effective Capacitability Theorem] \label{thm-eff-frostman}
	Suppose $x \in \Cant$ is strongly $h$-complex, where $h$ is a computable, convex order function. Then $x$ is $h$-capacitable.
\end{thm}

\begin{proof}
	By a theorem of \citet{kucera:1985} and \citet{gacs:1986} there exists a Martin-L\"of random real $R$ such that $x \leq_{\T} R$. In fact, there exists a Turing functional $\Phi$ and a $\Pi^0_1$ set $B \subseteq \Cant$ such that all elements of $B$ are Martin-L\"of random and for any $y \in \Cant$ there exists a $S \in B$ such that $y = \Phi(S)$.
		
	For every $\sigma \in \Str$ we define 
	$$
		\Pre(\sigma) = \{ \tau : \: \Phi(\tau) \Sgeq \sigma \: \& \: \forall \tau' \Sle \tau (\Phi(\tau') \nsupseteq \sigma) \}.
	$$ 
	$\Leb(\Pre(.))$ is an enumerable semimeasure. 
		
	It follows that $\Leb(\Pre(.))$ is multiplicatively dominated by an optimal enumerable semimeasure $\overline{M}$. There exists a constant $c$ such that for every $\sigma$,
	\[
		\Leb(\Pre(\sigma)) \leq c \overline{M}(\sigma).
	\]
	
	Since $x$ is strongly $\Hm{h}$-complex, there exists a constant $c'$ such that for all $n$,
	\[
		-\log \overline{M}(x\Rest{n}) \geq h(n) - c'.
	\]
	We conclude that there exists a constant $c''$ such that for all $n$,
	\[
		\Leb(\Pre(x\Rest{n})) \leq c \overline{M}(x\Rest{n}) \leq c'' 2^{-h(n)}
	\]
	Now consider the co-r.e.\ tree
	\[
		T = \{\sigma \in \Str: \: \text{for all $n \leq |\sigma|$, } \Leb(\Pre(\sigma\Rest{n})) \leq c''\, 2^{-h(n)} \}.
	\]
	From the above, $x \in [T]$, i.e.\ $x$ is an infinite path through $T$.
	
	Next we define a set of probability measures $B$ by transforming Lebesgue measure via $\Phi$.
	
	We want to preserve the randomness of $R$ when transforming with $\Phi$. Therefore, we require for every $\mu \in B$ that
	\begin{equation} \label{equ-lower_bound_measure}
		(\forall \sigma\in T) \; [\Leb(\Pre(\sigma)) \leq \mu(\sigma)].
	\end{equation}
	This way, a possible $\mu$-test covering $x$ would ``lift'' to a $\Leb$-test covering $R$.
	Furthermore, we want to meet the requirement given by \eqref{equ-frostman-prop}:
	\begin{equation} \label{equ-upper_bound_measure}
		\mu(\sigma) \leq \gamma 2^{-h(|\sigma|)},
	\end{equation}
	for some constant $\gamma$.
	
	We first show that there exists such a probability measure $\mu$ (and a suitable constant $\gamma$). This is achieved by the following lemma.

	\begin{lem} \label{lem-meas-along-tree}
		Suppose $T \subseteq \Str$ is a non-empty tree. Suppose further that $\eta$ is a semimeasure on $\Str$, and that there exist a convex order $h$ and a constant $\gamma$ such that $\eta(\tau) \leq \gamma\, 2^{-h(\tau)}$ for all $\tau \in T$. Then there is a probability measure $\mu$ such that for all $\sigma \in \Str$,
	\begin{equation} \label{equ-consist-req1}
		\mu(\sigma) \leq \gamma \:2^{-h(|\sigma|)}.
	\end{equation}
	and for all $\tau \in T$, 
	\begin{equation} \label{equ-consist-req2}
		\eta(\tau) \leq \mu(\tau). 
	\end{equation}	
	\end{lem}
	
	\begin{proof}
		We construct a measure inductively along $T$. Let $\mu(\Estr) = 1$. Given $\mu(\sigma)$ and both $\sigma\Conc 0$ and $\sigma\Conc 1$ are in $T$, we want to define $\mu(\sigma \Conc 0)$ and $\mu(\sigma \Conc 1)$ such that
		\[
			\mu(\sigma) = \mu(\sigma \Conc 0) + \mu(\sigma \Conc 1),
		\]
		and such that
		\[
			\eta(\sigma \Conc i) \leq \mu(\sigma \Conc i) \leq \gamma\, 2^{-h(|\sigma|+1)}.
		\]
		Such $\mu(\sigma \Conc 0)$ and $\mu(\sigma \Conc 1)$ exist for the following reason. We have
		\[
			\eta(\sigma\Conc 0) + \eta(\sigma\Conc 1) \leq \eta(\sigma) \leq \mu(\sigma),
		\] 
		by the assumption that $\eta$ is a semimeasure and the inductive hypothesis. 
		On the other hand, since $h$ is a convex order,
		\[
			2^{-h(|\sigma|+1)} + 2^{-h(|\sigma|+1)} \geq 2\cdot 2^{-h(|\sigma|)-1} = 2^{-h(|\sigma|)} \geq (1/\gamma)\,\mu(\sigma).
		\]
		Since the mapping from $[0,1]\times [0,1]$ to $\Real$ given by 
		\[
			(s,t) \; \mapsto \; \eta(\sigma\Conc 0) + s[\gamma\, 2^{-h(|\sigma|+1)} - \eta(\sigma\Conc 0)] + \eta(\sigma\Conc 1) + t[\gamma\, 2^{-h(|\sigma|+1)} - \eta(\sigma\Conc 1)]
		\]
		is continuous, the assertion follows from the intermediate value theorem.

	 	If either $\sigma\Conc 0$ or $\sigma\Conc 1$ is not in $T$, say $\sigma\Conc 1$, let $\mu(\sigma \Conc 0) = \gamma\, 2^{-h(|\sigma|+1)}$ and $\mu(\sigma \Conc 1) = \mu(\sigma) - \gamma\, 2^{-h(|\sigma|+1)}$.
	
	 	If neither $\sigma\Conc 0$ nor $\sigma\Conc 1$ is  in $T$ we let $\mu(\sigma\Conc 0) = \mu(\sigma\Conc 1) = \mu(\sigma)/2$, i.e.\ if we are not on the tree any longer, we distribute the mass uniformly henceforth. Requirement (\ref{equ-consist-req1}) then follows from the convexity of $h$.
	 \end{proof}

	Browsing through the tree $T$ as in the proof of the preceding lemma, one can effectively exclude (using the effectiveness of the metric $d_{\meas}$) more and more basic open neighborhoods from $B$. Each new level of $T$, and the enumeration of $\Pre$, refines the consistency requirements \eqref{equ-consist-req1} and \eqref{equ-consist-req2}. Hence there exists a $\Pi^0_1$ set $P_M$ of (representations of) measures satisfying \eqref{equ-lower_bound_measure} and \eqref{equ-upper_bound_measure}. Furthermore, Lemma \ref{lem-meas-along-tree} shows that $P_M$ is not empty.
		
	Finally, we show that there exists an element $p_\mu \in P_M$ such that $x$ is $\mu$-random relative to representation $p_\mu$. By Theorem \ref{thm-rand-conserv-basis} there exists a $p_\mu \in P_M$ such that $R$ is $\Leb$-random relative to $p_\mu$ (as a parameter of relative randomness, not as a measure). We claim that $x$ is $\mu$-random relative to representation $p_\mu$. Assume $W= (W_n)$ is a $\mu$-test (hence r.e.\ in $p_\mu$) that covers $x$. We define another test $V= (V_n)$: We start enumerating $\Pre(\sigma)$ into $V_n$ if and only if $\sigma$ is enumerated into $W_n$, and enumerate $\Pre(\sigma)$ into $V_n$ as long as $\sigma$ is not removed from $T$, i.e.\ as long as $\Leb(\Pre(\sigma)) \leq c''\, 2^{-h(|\sigma|)}$.
	
	It follows from the definition of $P_M$ that $V$ is a $p_\mu$-test (again, as a parameter, not a measure) that is correct for $\Leb$. Furthermore, $V$ covers $R$. But $R$ is $\Leb$-random relative to $p_\mu$, contradiction. This completes the proof of Theorem \ref{thm-eff-frostman}.
	 \end{proof}

Theorem \ref{thm-eff-frostman} can be relativized in a straightforward way. Suppose $h: \Nat \to \Nat$ is an arbitrary convex order. We say a real $x$ is \emph{strongly $h$-complex relative to $z \in \Cant$} if 
\[
	(\forall n)\: [ -\log \overline{M}^{h\oplus z}(x\Rest{n}) \geq h(n)],
\]
where $\overline{M}^{h\oplus z}$ is a optimal among all semimeasures enumerable in $h\oplus z$.

\begin{cor*}[Kjos-Hanssen and Reimann] \label{cor-relative-frostman} 
	Suppose $x$ is strongly $h$-complex relative to $z \in \Cant$, where $h$ is a convex order function.  Then there exists an $h$-bounded measure $\mu$ such that $x$ is $\mu$-random relative to $h\oplus z$.
\end{cor*}

We now can immediately deduce a sufficient criterion for continuous randomness.

\begin{cor*}
	Suppose $x$ is strongly $h$-complex, where $h$ is a convex order function with $h \leq_{\T} x$. Then $x$ is random for a continuous probability measure.
\end{cor*}

%
%
\section{Consequences of effective capacitability} \label{sec-consequ-cap}

%
%
\subsection{A new proof of Frostman's Lemma} \label{ssec-new-frost}

Theorem \ref{thm-eff-frostman} yields a new proof of \emph{Frostman's Lemma} \citep{frostman:1935}. 

\begin{thm} \label{thm-frostman}
	If $A$ is a compact subset of $\Cant$ with $\Hm{s}(A) > 0$, then there exists a probability measure $\mu$ such that $\supp(\mu) \subseteq A$, and such that there exists a constant $\gamma$ such that for all $\sigma \in \Str$,
	\[
		\mu(\sigma) \leq \gamma 2^{-|\sigma|s}.
	\]
\end{thm}

The problem is to make the support of $\mu$ contained in $A$ while at the same time respect the upper bound on the basic open sets. There are several known proofs of Frostman's Lemma. It can be obtained as a consequence of the existence of compact subsets with finite $\Hm{s}$-measure (see \citep{falconer:1990}), which is even more difficult to prove. Other proofs either construct a sequence of measures such that any limit point in the weak topology (which exists by compactness) will have the desired properties (see \citep{mattila:1995}). Alternatively, one can use the MaxFlow-MinCut Theorem (as in \citep{moerters-peres:ip}). A different approach, which works in general metric spaces, was given by \citet{howroyd:1994}, introducing \emph{weighted Hausdorff measures} and using the Hahn-Banach Theorem.
These proofs make essential use of compactness.

The new proof given here is an easy consequence of the effective capacitability theorem.

\begin{proof}[New proof of Theorem \ref{thm-frostman}]

Let $A \subseteq \Cant$ be closed with $\Hm{s}(A) > 0$. It follows there exists a $z \in \Cant$ such that $A$ is $\Pi^0_1(z)$. Since $\Hm{s}(A) > 0$, there exists an $x \in A$ that is strongly $\Hm{s}$-complex relative to $s\oplus z$. By Corollary \ref{cor-relative-frostman} there exists a probability measure $\mu$ such that $x$ is $\mu$-random relative to $s\oplus z$ and $\mu$ is $s$-bounded with constant $\gamma$.
	
	Since $A$ is $\Pi^0_1(z)$ and contains a $\mu$-$z$-random real, it follows that $\mu(A) > 0$. We restrict $\mu$ to $A$, written $\mu\Rest{A}$, by putting $\mu\Rest{A}(X) = \mu(X \cap A)$. This yields a measure with support contained in $A$ (note that $A$ is closed). Finally, normalizing
	\[
		\nu = \frac{\mu\Rest{A}}{\mu(A)}
	\]
	yields a probability measure $\nu$ satisfying \eqref{equ-frostman-prop} with constant $\gamma/\mu(A)$.  
 \end{proof}

Note that the new proof is of a profoundly effective nature. It features the Kucera-Gacs Theorem, which does not have a classical counterpart (the Lebesgue measure of an upper Turing cone over a non-recursive real is zero). Compactness is used in the form of a basis result for $\Pi^0_1$ classes. Finally, the problem of assigning non-trivial measure to $A$ is easily solved by making an element of $A$ random.

It seems that the general idea, using Lebesgue-random reals along with Turing or other reductions to construct measures, can be applied to other settings, too. Finally, it should be noted that the restriction to Cantor space is not actually a loss of generality, since it can be extended to other metric spaces, including Euclidean space, using \emph{net measures} (see for example \citep{rogers:1970}).

%
%
\subsection{Effective capacities and randomness notions} \label{ssec-cap-hrand-notions}

Theorem \ref{thm-eff-frostman} was proved for strongly $\Hm{h}$-random reals. In the following we show that this assumption is necessary. The proof also yields that in a reasonable sense, strong randomness is the ``strongest'' possible randomness notion for Hausdorff measures based on uniformly enumerable tests. As a corollary, we get that any dimension concept induced by such a randomness notion will coincide with the standard effective Hausdorff dimension.

\subsubsection*{Randomness notions} 
We give here a crude definition of a randomness notion that is sufficient for our purposes. For an exhaustive treatment, one should probably reformulate randomness as a notion of forcing (similar to the correspondence between Solovay genericity and weak randomness), and compare the strength of the forcing notions. However, in view of the results presented here, such an effort does not seem justified, and might only obscure things.

To keep notation simple, we identify premeasures with their representation(s). 

\begin{defn} \label{def-rand-notion}
	A \emph{notion of randomness} $\mathcal{R}$ uniformly assigns to a premeasure $\rho$ a collection of sets $W \subseteq \Nat \times \Str$ r.e.\ in $\rho$ such that $\bigcap W_n$ is a $\rho$-nullset.
	
	More precisely, a notion of randomness $\mathcal{R}$ is a set $\mathcal{R}(\rho) \subseteq \Nat$ definable in second order arithmetic from a parameter $\rho$ such that
	\[
 	 	e \in \mathcal{R}(\rho) \quad \Rightarrow \quad  \bigcap_n W^{\rho}_{e,n} \text{ is $\rho$-null},
	\] 
	where $W^\rho_e$ is the $e$-th r.e.\ in $\rho$ subset of $\Nat \times \Str$ under a canonical enumeration, and $W_{e,n} = \{ \sigma : \: (n, \sigma) \in W_e \}$. 
	Furthermore, we require a certain \emph{uniform consistency} between tests: There exists a (partial) recursive function $f$ such that
	\begin{multline*}
		 	 	\text{if $\rho,\eta$ are premeasures, } e \in \mathcal{R}(\rho), \text { $\eta \leq \rho$ pointwise and $\rho \leq_{\T} \eta$,}  \\
		\text{then } f(e)\downarrow\: \in \mathcal{R}(\eta) \text{ and } W^\rho_e = W^\eta_{f(e)}.
	\end{multline*}

	A real $x$ is \emph{$\rho$-random for notion $\mathcal{R}$} if for all $e \in \mathcal{R}(\rho)$, $x \not\in \bigcap_n W^{\rho}_{e,n}$.
\end{defn}

We consider the following examples. It is easy to see that they all satisfy the consistency requirement, since any $\rho$-test is also an $\eta$-test for any $\eta \leq \rho$ with $\rho \leq_{\T} \eta$.

\begin{enumerate}[(a)]
	\item \emph{Martin-L\"of randomness.} This is given by 
	\[
		\mathcal{R}_{ML}(p_\rho,k) \: \equiv \: \operatorname{Premeasure}(\rho) \; \& \; (\forall n) \left[ \sum \{\rho(\sigma) \colon \sigma \in W^{\rho}_{k,n} \} \leq 2^{-n} \right].
	\]
	
	\item \emph{Solovay randomness} \citep{solovay:1975}. 
	\begin{multline*}
		\mathcal{R}_{S}(p_\rho,k) \: \equiv \: \operatorname{Premeasure}(\rho) \; \& \; (\forall n) \biggl[ W_{k,n+1} \subseteq W_{k,n} \; \& \\
		 |W_{k,n} \setminus W_{k,n+1}| \neq 0 \text{ and finite } \& \;  \sum \{\rho(\sigma) \colon \sigma \in W^{\rho}_{k,1} \} \leq 1 \biggr].	
	\end{multline*}
	
	\item \emph{Strong randomness} (\citet{calude-staiger-terwijn:2006}). Here we need to quantify over open sets. These can be coded as reals. Given $z \in \Cant$, let $V_z$ be the subset of $\Str$ given by $\sigma \in V_z \: \Leftrightarrow \: z(\sigma) = 1$.
	\begin{multline*}
	\mathcal{R}_{\operatorname{str}}(p_\rho,k) \: \equiv \: \operatorname{Premeasure}(\rho) \; \& \; (\forall n) (\forall z) \biggl[ (V_z  \subseteq W_{k,n} \; \& \; V_z \text{ prefix-free } ) \\
	 \Rightarrow \;  \sum \{\rho(\sigma) \colon \sigma \in V_z \} \leq 2^{-n} \biggr].				
	\end{multline*}
	
	\item \emph{Vehement randomness} (\citet{kjoshanssen:ip}). 
	\begin{multline*}
		\mathcal{R}_{v}(p_\rho,k) \: \equiv \: \operatorname{Premeasure}(\rho) \; \& \; (\forall n) (\exists z) \biggl[ \Acyl{V_z} \supseteq  \Acyl{W_{k,n}} \; \& \; \\
		 \sum \{\rho(\sigma) \colon \sigma \in V_z \} \leq 2^{-n} \biggr].				
	\end{multline*}
	
\end{enumerate}

\subsubsection*{Comparison of randomness notions}
Given a premeasure $\rho$, we say a notion of randomness $\mathcal{R}_0$ is \emph{as strong as} a notion $\mathcal{R}_1$ on $\rho$, written $\mathcal{R}_0 \succeq_\rho \mathcal{R}_1$ every $\rho$-random for $\mathcal{R}_0$ is also $\rho$-random for $\mathcal{R}_1$. For a set $\Gamma$ of premeasures, we define $\mathcal{R}_0 \succeq_\Gamma \mathcal{R}_1$ if and only if for all $\rho \in \Gamma$, $\mathcal{R}_0 \succeq_\rho \mathcal{R}_1$. Finally, we write $\mathcal{R}_0 \equiv_{\Gamma} \mathcal{R}_1$ if 
$\mathcal{R}_0 \succeq_\Gamma \mathcal{R}_1$ and $\mathcal{R}_0 \preceq_\Gamma \mathcal{R}_1$.

Recall that $\Pmeas$ denotes the set of all probability measures, $\Hmeas$ the set of convex Hausdorff premeasures, and $\mathcal{G}$ the set of all geometrical premeasures.

It is easy to see that 
\[
	\mathcal{R}_{ML} \preceq_{\Pmeas \, \cup\, \Hmeas} \mathcal{R}_{S}.
\]
However, \citet{reimann-stephan:2005} showed that for all computable geometrical premeasures $\rho$ for which (G3) holds for some $q > 1$,
\[
	\mathcal{R}_{ML} \nsucceq_{\rho} \mathcal{R}_{S}.
\] 
Furthermore, they showed that
\[
	\mathcal{R}_{S} \preceq_{\mathcal{G}} \mathcal{R}_{\operatorname{str}},
\]
and for any computable, length-invariant, geometrical premeasure $\rho$,
\[
	\mathcal{R}_{S} \nsucceq_{\rho} \mathcal{R}_{\operatorname{str}}.
\] 

By noting that every open covering in $\Cant$ has a prefix-free subcovering, it follows that 
\begin{equation} \label{equ-str-leq-veh}
	\mathcal{R}_{\operatorname{str}} \preceq_{\Pmeas \, \cup\, \Hmeas} \mathcal{R}_{v}.
\end{equation}

\medskip
\citet{calude-staiger-terwijn:2006} showed that for a computable order $h$,
\[
	x \text{ is strongly $h$-complex } \quad \text{iff} \quad x \text{ is strongly $\Hm{h}$-random.}
\]
It follows that every strongly $\Hm{h}$-random real is $h$-capacitable.

We now use the effective capacitability theorem to show that strong randomness $\mathcal{R}_{\operatorname{str}}$ is the strongest possible randomness notion among a family of ``reasonable'' (in the sense of Definition \ref{def-rand-notion}) randomness concepts. Let $\Hmeas^*$ denote the family of all computable convex Hausdorff premeasures.

\begin{thm} \label{thm-equiv-strong-rand}
	Suppose $\mathcal{R}$ is a randomness notion such that $\mathcal{R}_{\operatorname{str}} \preceq_{\Hmeas^*} \mathcal{R}$ and $\mathcal{R}_{\operatorname{str}} \equiv_{\Pmeas} \mathcal{R}$. Then  $\mathcal{R}_{\operatorname{str}} \equiv_{\Hmeas^*} \mathcal{R}$.
\end{thm}
    
\begin{proof}
	Let $x$ be strongly $\Hm{h}$-random for computable $h$.  The effective capacitability theorem implies that $x$ is $\mu$-random for some $h$-bounded probability measure $\mu$. Assume $x$ is not $\Hm{h}$-random for $\mathcal{R}$. Then there exists an $\Hm{h}$-test $W$ for $\mathcal{R}$ that covers $x$. It follows directly from the fact that $\mu$ is $h$-bounded and the consistency property of the randomness notion $\mathcal{R}$ that $W$ is also a $\mu$-test for $\mathcal{R}$. Thus $x$ is not $\mu$-random for $\mathcal{R}$, contradicting the assumption $\mathcal{R}_{\operatorname{str}} \equiv_{\Pmeas} \mathcal{R}$.
 \end{proof}

Hence any randomness notion that is as strong as strong randomness but that coincides with strong randomness on probability measures actually coincides with strong randomness on all computable Hausdorff premeasures.

\begin{cor*} \label{cor-equiv-strong-veh}
	$\mathcal{R}_{\operatorname{str}} \equiv_{\Pmeas \, \cup\, \Hmeas^*} \mathcal{R}_{v}$, that is for probability and computable Hausdorff measures, strong and vehement randomness coincide.
\end{cor*}

\begin{proof}
	By Theorem \ref{thm-equiv-strong-rand} and \eqref{equ-str-leq-veh} it suffices to prove that $\mathcal{R}_{\operatorname{str}} \equiv_{\Pmeas} \mathcal{R}_{v}$. We will actually show that $\mathcal{R}_{\operatorname{ML}} \equiv_{\Pmeas} \mathcal{R}_{v}$.

	We need the following Lemma.
	
	\begin{lem}
		If $W \subseteq \Str$ is r.e.\, then there exists a prefix-free r.e.\ set $U \subseteq \Str$ such that $\Acyl{U} = \Acyl{W}$.
	\end{lem}
	
	\begin{proof}[Proof of Lemma] 
		We enumerate a set $U \subseteq \Str$ as follows. Whenever $\sigma$ is enumerated into $W$, check whether an initial segment of $\sigma$ has already been enumerated into $U$. If so, do nothing. Otherwise, if no extension of $\sigma$ has been enumerated into $U$ at this time, enumerate $\sigma$ into $U$. If extensions of $\sigma$ have been enumerated into $U$ already, say $\tau_1, \dots, \tau_k$, choose strings $\xi_1, \dots, \xi_m$ such that $S = \{\tau_1, \dots, \tau_k,\xi_1, \dots, \xi_m\}$ is prefix-free and $\Cyl{S} =\Cyl{\sigma}$. Now it is easy to see that $U$ has the desired properties. 
	 \end{proof}
	
	To complete the proof of Corollary \ref{cor-equiv-strong-veh}, let $\mu$ be a probability measure and suppose $W$ is a vehement $\mu$-test. Let $(V_n)$ be a sequence of open sets such that for each n, $\Acyl{V_n} \supseteq \Acyl{W_n}$ and $\sum \{\mu(\sigma)\colon \sigma \in V_n\} \leq 2^{-n}$. By the previous lemma, for each $n$ there exists (uniformly in $n$) a prefix-free r.e.\ set $U_n \subseteq \Str$ such that $\Acyl{W_n} = \Acyl{U_n}$. Note that for a prefix-free set $U \subseteq \Str$, $\mu(\Cyl{U}) = \sum \{\mu(\sigma)\colon \sigma \in U\}$. Hence for all $n$,
	\[
		\sum \{\mu(\sigma)\colon \sigma \in U_n\} = \mu(\Acyl{U_n}) \leq \mu(\Acyl{V_n}) \leq \{\mu(\sigma)\colon \sigma \in V_n\}  \leq 2^{-n},   
	\] 
	and thus $U = \{U_n\}$ defines a Martin-L\"of $\mu$-test that covers the same reals that $W$ does.
 \end{proof}

Finally we note that a similar argument yields that strong randomness is a necessary condition for effective capacitability.
This was observed by Bj{\o}rn Kjos-Hanssen \citep{kjos-hanssen:email}.

\begin{cor*}[Kjos-Hanssen]
	If $x$ is not strongly $\Hm{h}$-random then $x$ is not effectively $h$-capacitable.
\end{cor*}

\begin{proof}
	Assume $x$ is not strongly $\Hm{h}$-random. Let $\mu$ be an $h$-bounded probability measure. Let $W = \{W_n\}$ be a vehement $\Hm{h}$-test covering $x$, and let $V$ be such that for each n, $\Acyl{V_n} \supseteq \Acyl{W_n}$ and $\sum \{\mu(\sigma)\colon \sigma \in W_n\} \leq 2^{-n}$. As in the proof of Corollary \ref{cor-equiv-strong-veh}, let $U_n$ be a prefix-free r.e.\ set generating the same open set as $W_n$ does. Then
	\[
		\sum \{\mu(\sigma)\colon \sigma \in U_n\} = \mu(\Acyl{U_n}) \leq \mu(\Acyl{V_n}) \leq \{\mu(\sigma)\colon \sigma \in V_n\}  \leq 2^{-n},
	\]
	and hence $U = \{U_n\}$ forms a $\mu$-test covering $x$. Thus $x$ cannot be $\mu$-random.
 \end{proof}

%
%
\subsection{Effective capacities and effective dimension} \label{ssec-cap-dim}
	
Given a randomness notion $\mathcal{R}$, one can define a corresponding \emph{effective Hausdorff dimension} by putting
\[
	\Hdim[\mathcal{R}](x) = \inf\{s \in \Rat\colon x \text{ is  not $\Hm{s}$-random for $\mathcal{R}$} \}.
\]		
We refer to \citep{falconer:1990, mattila:1995, lutz:2003, reimann:2004, hitchcock-lutz-mayordomo:2005} for definitions and background on classical and effective dimension concepts. 
Although it has been shown in \citep{reimann-stephan:2005} that various randomness notions do not agree on Hausdorff measures and yield a strict hierarchy of randomness concepts, we will see now that they all yield the same dimension concept. Furthermore, we use the effective capacitability theorem to prove a new characterization of effective dimension.

\begin{thm}
	For any real $x$,
	\[
		\Hdim[\mathcal{R}_{\operatorname{ML}}](x) = \Hdim[\mathcal{R}_S](x) = \Hdim[\mathcal{R}_{\operatorname{str}}](x) = \Hdim[\mathcal{R}_{v}](x).
	\]
\end{thm}

\begin{proof}
	By the results of the previous subsection, it suffices to show that for any $s < t$,
	$x$ not being strongly $\Hm{s}$-random implies $x$ not being (Martin-L\"of) $\Hm{t}$-random.
	
	Suppose $W = \{W_n\}$ is a strong $\Hm{s}$-test. We show that 
	\[
		\sum \{2^{-|\sigma|t} \colon \sigma \in W_n\}
	\]
	can be uniformly bounded by a geometric sequence effectively converging to $0$. This can be used to enumerate $W$ as a Martin-L\"of $\Hm{t}$-test.
	 
	Partition $W_n$ into ``prefix-levels'' $W^{(0)}_n, W^{(1)}_n, \dots$, where 
	$\sigma$ is in $W^{(k)}_n$ if and only if there are $k$ proper prefixes of $\sigma$ in $W$.
	
	Then 
	\[
		\sum \{2^{-|\sigma|t} \colon \sigma \in W^{(0)}_n\} \leq \sum \{2^{-|\sigma|s} \colon \sigma \in W^{(0)}_n\} \leq 2^{-n},
	\]  
	and in general, since every string in $W^{(k)}_n$ has length at least $k$,
	\[
		\sum \{2^{-|\sigma|t} \colon \sigma \in W^{(k)}_n\} \leq  \sum \{2^{-|\sigma|s} \, 2^{-k(t-s)} \colon \sigma \in W^{(k)}_n\} \leq 2^{-n}\, 2^{-k(t-s)}.
	\]
	Hence
	\[
		\sum \{2^{-|\sigma|t} \colon \sigma \in W_n\} \leq 2^{-n} \sum_k 2^{-k(t-s)} = 2^{-n} \frac{1}{1 - 2^{-(t-s)}}.
	\]
 \end{proof}

This answers a question by Kjos-Hanssen. The result was independently obtained by \citet{kjos-hanssen:draft} and \citet{miller:draft}. 

Finally, we use the effective capacitability theorem to give a new characterization of effective dimension, which reflects the duality between the complexity of a real and the capacity of a measure making it random.
  
\begin{thm}
	For any real $x \in \Cant$,
	\begin{equation} \label{equ-cap-edim}
		\Hdim[\mathcal{R}_{\operatorname{ML}}] (x) = \sup \{ s\in \Rat :  x \text{ is $h$-capacitable for $h(n) = sn$} \},
	\end{equation}
	where the supremum is assumed to be zero if the set is empty.
\end{thm}

\begin{proof}
	Denote the set on the right hand side of \eqref{equ-cap-edim} by $C$. Obviously, $C$ is the intersection of $\Rat$ with an interval. If $\Hdim[\mathcal{R}_{\operatorname{ML}}] (x) \geq s >0$, then by Theorem \ref{thm-eff-frostman} $s \in C$. Hence $\sup C \geq \Hdim[\mathcal{R}_{\operatorname{ML}}] (x)$. On the other hand, suppose $t > \Hdim[\mathcal{R}_{\operatorname{ML}}] (x)$ and $\mu$ is a probability measure such that for all $\sigma$, $\mu(\sigma) \leq \gamma \, 2^{-t|\sigma|}$. Since $t > \Hdim[\mathcal{R}_{\operatorname{ML}}] (x)$, there exists a $\Hm{t}$-test that covers $x$. But every $\Hm{t}$-test is also a $\mu$-test. Hence, $t \not\in C$, which implies $t \geq \sup C$ and therefore $\Hdim[\mathcal{R}_{\operatorname{ML}}] (x) \geq \sup C$.
 \end{proof}

%
%
\section{Concluding remarks and open questions} \label{sec-conclusion}
	 
The paper has established an exact correspondence between randomness for Hausdorff measures and for probability measures, as inspired by the classical theory which has proved extremely helpful in geometric measure theory. We believe that the effective relation is useful to a similar extent when dealing with randomness for Hausdorff measures, since probability measures are usually nicer to deal with.

Frostman's Lemma has been extended to packing measures by \citet{cutler:1995}.
For details on packing measures and packing dimension, see \citep{falconer:1990}.

\begin{thm} \label{thm-frost-pack}
	If $A$ is a compact subset of $\Cant$ with $\dim_{\operatorname{P}}(A) \geq t > 0$, then there exists a probability measure $\mu$ such that $\supp(\mu) \subseteq A$, and for each $x \in \Cant$ it holds for infinitely many $n$ that
	\[
		\mu(x\Rest{n}) \leq \gamma 2^{-nt}.
	\]
\end{thm}

The effective analogue of this does not hold, in fact it fails in a striking way. Recently, \citet{conidis:ta} has constructed a countable $\Pi^0_1$-class of effective packing dimension $1$. By an observation of \citet{kjoshanssen-montalban:2005}, no member of a countable $\Pi^0_1$-class can be random for a continuous measure, in particular, it cannot be random for a measure as in Theorem \ref{thm-frost-pack}.

This adds to the evidence that non-trivial effective packing dimension cannot really be seen as an indicator of random content, whereas non-trivial effective Hausdorff dimension can, although the computational properties of such reals can be quite different from Martin-L\"of random reals (see recent work by \citet{miller:draft} and \citet{greenberg-miller:draft}).

\subsection*{Open questions}

Recent breakthroughs by \citet{miller:draft} and \citet{greenberg-miller:draft} on the so-called \emph{dimension problem} have shown that not every $\Hm{h}$-random real computes a $\lambda$-random real, even when $h$ is a well-behaved order function such as $h(n) = tn$, where $t$ can be arbitrarily close to $1$. In light of these results it would be very interesting to have a classification of the reals which \emph{do} compute a $\lambda$-random real. One way to ensure this is to be random for a measure which the real itself computes (by results of Levin \citep{zvonkin-levin:1970} and \citet{kautz:1991}). 

\textbf{Question.} \emph{For which strongly $\Hm{h}$-random $x$ does there exist a continuous measure $\mu \leq_{\T} x$ such that $x$ is $\mu$-random?}

It follows directly from the construction in the proof of Theorem \ref{thm-eff-frostman}     that the complexity of $\mu$ does not exceed that of $x$ by much.

\begin{cor*}
	If $x$ is strongly $\Hm{h}$-random, where $h$ is a computable order function, then $x$ is random for a continuous probability measure $\mu$ such that $p_\mu \leq_{\T} x'$.
\end{cor*}

But besides the usual basis theorems for $\Pi^0_1$-classes, we do not seem to have much more control over $\mu$.

It is known that $\mu$ cannot always be very simple.

\begin{thm}[\citet{reimann:2004}]
	For every computable convex order function $h$ there exists a real $x$ such that $x$ is $\Hm{h}$-random but not random for any computable probability measure.
\end{thm}

Finally, it would be a fascinating yet probably difficult endeavor to investigate how much of the results (also those from \citep{reimann-slaman:ip2}) could be transfered to a \emph{representation independent setting} using Miller's framework of \emph{degrees of continuous functions} \citep{miller:2004}.

%
%
\section{Acknowledgements} \label{sec-acknow}

The author would like thank an anonymous referee for many extremely helpful and relevant comments.  The author is also deeply indebted to Bj{\o}rn Kjos-Hanssen for numerous discussions and suggestions, as well as for discovering an error in an earlier version of this article.

\end{document}